\documentclass[10pt]{article}

\usepackage{graphicx}
\usepackage{color}
\usepackage{framed}
\usepackage{alltt}
\usepackage[T1]{fontenc}
\usepackage{natbib}
\usepackage{amsmath}
\usepackage{amsthm}
\usepackage{amssymb}
\usepackage{mathtools}
\usepackage{authblk}
\usepackage[ruled]{algorithm2e}
\usepackage{mathptmx}

\newtheorem{theorem}{Theorem}[section]
\newtheorem{corollary}[theorem]{Corollary}
\newtheorem{proposition}[theorem]{Proposition}
\newtheorem{lemma}[theorem]{Lemma}
\theoremstyle{definition}
\newtheorem{definition}[theorem]{Definition}
\newtheorem{example}[theorem]{Example}
\DeclareMathOperator*{\argmax}{argmax}

\DeclareMathOperator*{\pr}{pr}

\newcommand{\UU}{\mathcal U}
\newcommand{\CC}{\mathcal C}
\newcommand{\R}{\mathbb R}
\newcommand{\vC}{{\v{C}}}
\newcommand{\cech}{{\v{C}}ech }
\newcommand{\field}{\mathbb F}
\begin{document}

\title{Divisive cover}
\author[1]{Blaser, Nello \thanks{nello.blaser@uib.no}}
\author[1]{Brun, Morten \thanks{morten.brun@uib.no}}
\affil[1]{Department of Mathematics, University of Bergen}

\maketitle

\begin{abstract}

The aim of this paper is to present a method for computing
persistent homology that performs well at large filtration values. 
To this end we introduce the concept of filtered covers. 
Given
a parameter \(\delta\) with \(0 < \delta \le 1\) we introduce the
concept of a $\delta$-filtered cover and show that its filtered nerve
is interleaved with the {\v{C}}ech complex. 
We introduce a
particular $\delta$-filtered cover, the divisive cover. The special
feature of the divisive cover is that it is constructed top-down. If
we disregard fine scale structure and \(X\) is a finite subspace of
Euclidean space, then we obtain
a filtered simplicial complex whose size makes computing
persistent homology feasible.
\end{abstract}

\section{Introduction}

The concept of persistent homology was introduced in  \cite{Edelsbrunner2000} and has since been used in a wide
range of applications. The persistent homology of a finite metric
space $X$ can be approached by using several different constructions
of filtered simplicial complexes, such as the {\v{C}}ech complex,
Vietoris-Rips complex or witness complex. Several approximations
of the Vietoris-Rips complex have recently been proposed to speed up
calculations \cite{Sheehy2013,MR3394707,simba}. 


In this paper we construct a new approximation to the {\v{C}}ech
complex computing persistent homology down to a
predefined threshold that can be chosen arbitrarily. The
complexity of our algorithm grows with the
ratio between the radius of \(X\) and the threshold. We also present a
version with theoretical guarantees on size and time. 
If \(X\) is a subset of \(d\)-dimensional Euclidean space then the
size of our approximation is bounded by an upper bound that is
independent of the
cardinality \(n\) of \(X\) and the required computation time 
is linear in \(n\).
However the constants are so big that this is no improvement
in practice. 

The method presented here is fundamentally different from existing
algorithms for persistent homology.  Most approximations to the
Vietoris-Rips complex are fundamentally bottom-up \cite{Hudson,
  Sheehy2013, MR3394707}, whereas our approach is top-down.

We introduce the notion of filtered nerve of a filtered cover and we
give an example of a filtered cover whose filtered nerve has a
filtered chain complex which is computationally tractable. Moreover we
show that the resulting nerve is interleaved with the {\v{C}}ech nerve
in a multiplicative sense, similar to the Vietoris-Rips complex. In
Section~\ref{sec:filtered} we introduce the notion of filtered and
$\delta$-filtered covers and show that \(\delta\)-filtered covers are
interleaved with the {\v{C}}ech filtration. Section~\ref{sec:divisive}
introduces divisive covers, a particular class of $\delta$-filtered
covers. Complexity estimates for divisive covers are presented in
Section~\ref{sec:complexity}. In Section~\ref{sec:examples} we show
how divisive covers can be applied to synthetic and to real world data
sets and in Section~\ref{sec:conclusion} we discuss our results.


\section{Filtered covers} \label{sec:filtered}
We introduce the notions of a filtered and $\delta$-filtered cover of
a bounded metric space. Throughout this section \(X = (X,d)\) will be
a fixed but arbitrary bounded metric space. First recall the definition of a cover. 
\begin{definition}
  A {\em cover} of \(X\) is a set \(\UU\) of subsets of \(X\) such that every point in \(X\) is contained in a member of \(\UU\).
\end{definition}
Recall that a {\em simplicial complex} \(K\) consists of a vertex set \(V\) and a set \(K\) of subsets of \(V\) with the property that if \(\sigma\) is a member of \(K\) and if \(\tau\) is a subset of \(\sigma\), then \(\tau\) is a member of \(K\). Also recall that every simplicial complex \(K\) has an underlying topological space \(|K|\). The book \cite{MR2766102} may serve as gentle introduction and reference to abstract simplicial complexes.
\begin{definition}
  Let \(\UU\) be a cover of \(X\). The {\em nerve} \(N(\UU)\) of
  \(\UU\) is the simplicial complex with vertex set \(\UU\) defined as
  follows: A finite subset \(\sigma = \{U_0,\dots,U_n\}\) of \(\UU\)
  is a member of \(N(\UU)\) if and only if the intersection of \(U_0
  \cap \dots \cap U_n\) is non-empty. 
\end{definition}

Note that the nerve construction \(\UU \mapsto N(\UU)\) is functorial 
in the sense
that if \(\UU \subseteq \mathcal V\) is an inclusion of covers of
\(X\), then we have an induced inclusion \(N(\UU) \subseteq N(\mathcal
V)\) of nevers.
 
If \(B \subseteq A\) is an inclusion of partially ordered sets, 
we say that \(B\) is {\em cofinal} in \(A\) if for every \(a \in A\),
there exists \(b \in B\) so that \(a \le b\). Given a cover \(\UU\),
we consider it as a partially ordered set with partial order given by
inclusion. We will need the
following result several times: 
\begin{lemma}\label{cofinality}
  If \(\UU \subseteq \mathcal V\) are covers of \(X\) and if \(\UU\)
  is cofinal in \({\mathcal V}\), then the geometric
  realization of the inclusion \(N(\UU) \subseteq N(\mathcal V)\) is a
  homotopy equivalence. 
\end{lemma}
\begin{proof}
  Since \(\UU\) is cofinal in \(\mathcal V\) there exists a map \(f
  \colon \mathcal V \to \UU\) such that \(V \subseteq 
  f(V)\) for all \(V \in \mathcal V\). Note that the formula
  \(Nf(\{V_0,\dots,V_n\}) = \{f(V_0),\dots, f(V_n)\}\) defines a
  simplicial map \(Nf \colon N(\mathcal V) \to N(\UU)\). Similarly,
  the inclusion \(i \colon \UU \to \mathcal V\) induces a simplicial
  map \(Ni \colon N(\UU) \to N(\mathcal V) \). 

  Since \(V \subseteq f(V)\) for all \(V \in \mathcal V\), the
  composite \(Nf \circ Ni\) is contiguous with the identity map on
  \(N(\UU)\) in the sense that for every face \(\sigma\) of \(N(\UU)\),
  the set \(\sigma \cup (Nf \circ Ni(\sigma))\) is a face of
  \(N(\UU)\). It follows that the geometric realization of \(Nf \circ
  Ni\) is homotopic to the 
  identity on the geometric realization of \(N(\UU)\) \cite[Lemma 2
  p. 130]{Spanier}. Similarly the geometric realization of 
  \(Ni \circ N f\) is homotopic to the identity on the geometric
  realization of \(N(\mathcal V)\).

\end{proof}


We are now ready to define and establish some basic properties of filtered bases and filtered nerves.
\begin{definition}
  A {\em filtered basis of \(X\)} is a basis \(\UU\) for the metric topology on \(X\) with the property that \(X\) is a member of \(\UU\).
  Given \(t > 0\) we write \(\UU_t\) for the cover of \(X\)
  consisting of members of \(\UU\) contained in an open ball in \(X\) of radius \(t\).
\end{definition}


\begin{definition}
  Let \(\UU\) be a filtered basis of \(X\). The {\em filtered nerve} of \(\UU\) is the collection \(\{N(\UU_t)\}_{t >0}\) together with the inclusions \(N(\UU_s) \subseteq N(\UU_t)\) induced by the inclusions \(\UU_s \subseteq \UU_t\). 
\end{definition}
Since \(X\) is bounded there exists \(T > 0\) so that \(N(\UU_t) = N(\UU)\) for \(t \ge T\).

\begin{definition}\label{def:delta-filtered-cover}
  Let \(\UU\) be a filtered basis of \(X\) and let \(\delta\) be a parameter satisfying \(0 < \delta \le 1\). We say that \(\UU\) is a \(\delta\)-filtered basis of \(X\) if for every \(x \in X\) and every \(r > 0\), there exists a member \(A\) of \(\UU_r\) containing \(B(x,\delta r)\).
\end{definition}
\begin{example}\label{absolute {\v{C}}ech cover}
  The {\v{C}}ech cover \(\CC = \check\CC(X)\) 
consisting of all balls in \(X\) is
\(1\)-filtered. 
\end{example}
\begin{example}
  Let \(0 < \delta < 1\) and choose \(x \in X\) and \(R > 0\) so that
  \(X\) is contained in the open ball \(B(x,R)\). We claim that the
  subset \(\UU = \check\CC(X,\delta)\) of the \vC ech cover
  \(\check\CC(X)\) consisting of balls of radius \(\delta^k R\),
  where \(k\) is a nonnegative integer, is a \(\delta\)-filtered basis of
  \(X\). Indeed, 
  let \(k\) be the nonnegative integer with \(\delta^{k+1}R \le
  r \le \delta^{k} R\). Since \(\delta r \le \delta^{k+1} R\), the set
  \(B(p,\delta r)\) is contained in the member
  \(B(p,\delta^{k+1} R)\) of \(\UU_{\delta^{k+1} R}\) for
  every \(p \in X\). We can
  finish the argument by noting that since \(\delta^{k+1} R \le r\)
  the cover \(\UU_{\delta^{k+1} R}\) is a subcover of \(\UU_r\).  
\end{example}

We now introduce some notation regarding persistent homology. For the rest of this section \(\field\) will denote a fixed but arbitrary field. 

A {\em persistence module} \(V = (V_t)_{t > 0}\) consists of a
\(\field\)-vector space \(V_t\) for each positive real number \(t\) together with homomorphisms
\begin{displaymath}
  V_{s < t} \colon V_s \to V_t 
\end{displaymath}
for \(s < t\). These homomorphisms are subject to the condition that \(V_{s < t} \circ V_{r < s} = V_{r < t}\) whenever \(r < s < t\). Given \(\lambda_1, \lambda_2 \ge 1\), two persistence modules \(V\) and \(W\) are {\em multiplicatively \((\lambda_1,\lambda_2)\)-interleaved} if there exist \(\field\)-linear maps \(f_t \colon V_{t} \to W_{\lambda_1 t}\) and \(g_t \colon W_t \to V_{\lambda_2 t}\) for all real numbers \(t\) such that for all \(s < t\) the following relations hold
\begin{align*}
  f_{\lambda_2 t} \circ g_t &= W_{t < \lambda_1\lambda_2 t}, \\
  g_{\lambda_1 t} \circ f_t &= V_{t < \lambda_1\lambda_2 t}, \\
  g_{t} \circ W_{s < t} &= V_{\lambda_2 s < \lambda_2 t} \circ g_s \qquad \text{and} \\
  f_{t} \circ V_{s < t} &= W_{\lambda_1 s < \lambda_1 t} \circ f_s .
\end{align*}

Given a simplicial complex \(K\), we write \(H_*(K)\) for the homology
of \(K\) with coefficients in the field \(\field\).

The following example justifies working with the intrisic \cech complex
instead of the relative \cech complex.
\begin{example}\label{relative {\v{C}}ech cover}
  Let \(X\) be a subspace of a metric space \(M\), let \(\CC(X)\) be
  the filtered basis from Example \ref{absolute {\v{C}}ech
    cover}. 
  Let \(\CC(X,M)\) be the relative \cech cover
  consisting of balls in \(M\) with center in \(X\), that
  is, with
  \(\CC(X,M)_t\) consisting of balls in \(M\) with center in \(X\) of
  radius at most \(t\).

  The homology of the intrinsic \cech chain complex \(\CC_*(X)_t\) consisting of
  linear combinations of subsets \(\sigma \subseteq X\) with the
  property that \(\sigma \subseteq B(x,t)\) for some \(x \in X\) is
  isomorphic to the homology of \(N(\CC(X)_t)\).
  Similarly, the homology of 
  the ambient \cech chain complex \(\CC_*(X,M)_t\) consisting of
  linear combinations of subsets \(\sigma \subseteq X\) with the
  property that \(\sigma \subseteq B(p,t)\) for some \(p \in M\) is
  isomorphic to the homology of \(N(\CC(X,M)_t)\).
  By construction \(\CC_*(X)_t \subseteq
  \CC_*(X,M)_t\), and by the triangle inequality
  \(\CC_*(X,M)_t \subseteq
  \CC_*(X)_{2t}\). Thus, the
  persistent homology of \(N(\CC(X))\) is \((1,2)\)-interleaved with the
  persistent homology of \(N(\CC(X,M))\).

  By the Nerve Theorem \cite[Corollary 4G.3]{MR1867354},
  if all non-empty intersections of balls in \(M\) are contractible, the
  geometric realization of the nerve \(N(\CC(X,M)_t)\) of the cover
  \(\CC(X,M)_t\), consisting of balls in \(M\) with center in \(X\) of
  radius at most \(t\), 
  is homotopy equivalent to the
  union of all balls in \(M\) of radius \(t\) with center in \(X\). This
  is the interior of the \(t\)-thickening of 
  \(X\) in \(M\). 
\end{example}

\begin{theorem}[Relationship between $\delta$-filtered basis and {\v{C}}ech complex] \label{thm:cech}
  Let \(\CC\) be the \cech cover from Example \ref{absolute
    {\v{C}}ech cover}, let \(\UU\) be a \(\delta\)-filtered basis
  of \(X\) and \(N(\CC)\) and \(N(\UU)\) be their filtered
  nerves. Then the persistent homology of \(N(\UU)\) is
  multiplicatively \((1,1/\delta)\)-interleaved with the persistent
  homology of \(N(\CC)\). 
\end{theorem}
\begin{proof}
  By definition, the partially ordered set \( \CC_r\)
  is cofinal in \( \CC_r \cup \UU_r\)
  and \(\UU_r\) is cofinal in 
  \( \CC_{\delta r} \cup \UU_r\).
  Thus, by Lemma \ref{cofinality}, the homology \(H_*(N(\CC_r))\) is
  isomorphic to the homology \(H_*(N(\CC_r \cup \UU_r))\) and the
  homology \(H_*(N(\UU_r))\) is isomorphic to the homology \(H_*(N(
  \CC_{\delta r} \cup \UU_r ))\). Now the result follows from
  functoriality of the nerve construction by considering the composites
  \begin{displaymath}
    \CC_{\delta r} \cup \UU_r \subseteq \CC_{r} \cup \UU_{r} \subseteq \CC_{r} \cup \UU_{r/\delta}
  \end{displaymath}
  and
  \begin{displaymath}
    \CC_r \cup \UU_{r} \subseteq \CC_r \cup \UU_{r/\delta} \subseteq \CC_{r/\delta} \cup \UU_{r/\delta}. 
  \end{displaymath}
\end{proof}

An easy diagram chase now gives:
\begin{corollary}\label{stabilityOfPersistence}
  If \(t > 0\) can be chosen so that
  in the situation of Theorem \ref{thm:cech}, the \(\field\)-linear maps \(H_*(N(\CC_{\delta t < t}))\) and \(H_*(N(\CC_{t < t/\delta}))\) are both isomorphisms, then \(H_*(N(\CC_t))\) is isomorphic to the image of the homomorphism \(H_*(N(\UU_{t< t/\delta}))\).
\end{corollary}
In \cite[Theorem 1]{Chazal2005304} it is shown that if \(X\) is open
in \(M = \R^d\), then the conditions of Corollary
\ref{stabilityOfPersistence} are satisfied when \(2t/\delta\) is
smaller than the weak feature size of \(X\).
In \cite{Chazal2007}
these considerations have been extended to similar results when \(X\)
is a finite subset of a compact subset \(M\) of \(\R^d\). 
Moreover, \cite[Homological Inference Theorem]{MR2460372} shows similar results for the homological feature size of \(X\). 

Next we introduce \(\delta\)-filtered covers, which do not require the cover to be a basis.
\begin{definition}\label{define-delta-filtered-cover}
Let \(\UU\) be a cover of \(X\), and \(\delta\) and \(r\) be parameters satisfying \(0 < \delta \le 1\) and \(r \ge 0\). 
We say that \(\UU\) is a \em{\(\delta\)-filtered cover} of \(X\) of resolution \(r\) if there exists a filtered basis \(\mathcal V\) such that \(\UU_s\) is cofinal in \(\mathcal V_s\) for all \(s \ge r\).
\end{definition}

\begin{corollary}
Let \(X\) be a bounded metric space, \(r \ge 0\) and \(\UU\) and \(\mathcal V\) be as in Definition \ref{define-delta-filtered-cover}. Then the persistent homology of \(N\UU_t\) and the persistent homology of \(N\mathcal V_t\) are isomorphic for \(t \ge r\). 
\end{corollary}
\begin{proof}
This is a direct consequence of Lemma \ref{cofinality}.
\end{proof}


\section{Divisive Covers} \label{sec:divisive}

In this section we discuss an algorithm to construct a
$\delta$-filtered cover of a bounded metric space \(X\). First it
divides \(X\) into two smaller sets. It 
continues by dividing the biggest of the resulting two sets into two,
and then iteratively divides the biggest of the remaining sets in
two.  

In order to describe the algorithm, we first define diameter and relative radius of a subset of a metric space. 
 \begin{definition}
Let \(X\) be a metric space and let \(Y\) be a subset of \(X\). 
\begin{enumerate}
\item The {\em diameter} of \(Y\) is defined as
  \begin{displaymath}
    d(Y) = \sup\{d(y_1,y_2) \mid y_1,y_2 \in Y\}.
  \end{displaymath}
\item 
The {\em radius} of \(Y\) relative to \(X\) is defined as 
\begin{equation*}
	r(Y) = \inf \{r > 0 \mid Y \subseteq B(x, r) \text{ for some } x \in X \}
\end{equation*}
\end{enumerate}
 \end{definition}

\begin{definition}
  A {\em \(\delta\)-division} of a subset \(Y\) of radius \(r\)
  relative to a
  bounded metric 
  space \(X\) consists of a cover \(\{Y_1, Y_2\}\) of \(Y\) consisting
  of proper subsets of \(Y\) with the
  property that for every \(y \in Y\) the intersection \(Y \cap B(y, \delta r)\)
  is contained in at least one of the sets \(Y_1\) and \(Y_2\). 
\end{definition}
\begin{definition}
Let \(X\) be a metric space. A {\em \(\delta\)-divisive cover} of \(X\) of resolution \(r \ge 0\) is a
  cover \(\UU\) of \(X\) containing \(X\) and a \(\delta\)-division
  \(\{Y_1,Y_2\}\) of every \(Y \in \UU\) of radius \(r(Y)>r\).
\end{definition}
\begin{lemma}
  Let \(\UU\) be a \(\delta\)-divisive cover of resolution \(r \ge 0\) of a bounded metric space \(X\). If every non-empty subset of
  \(\UU\) has a minimal element with respect to inclusion, then \(\UU\) is a
  \(\delta\)-filtered cover of resolution \(r\). 
\end{lemma}
\begin{proof}
  Let \(x \in X\) and let \(s > r\). Let \(Y \in \UU\) be minimal
  under the condition that \(B(x, \delta s) \subseteq Y\). Suppose
  that \(r(Y)>s\) and let \(\{Y_1,Y_2\}\) be a \(\delta\)-division of
  \(Y\) contained in \(\UU\). Since \(B(x,\delta s) \subseteq
  B(x,\delta r(Y))\) we have that \(B(x,\delta s)\) is contained in
  either \(Y_1\) or \(Y_2\) and \(Y_1\) and \(Y_2\) are proper subsets
  of \(Y\). This contradicts the minimality of \(Y\).
\end{proof}
\begin{corollary}
If \(\UU\) is a finite \(\delta\)-divisive cover of \(X\), then \(\UU\) is a
\(\delta\)-filtered cover. 
\end{corollary}

There exist many ways to construct \(\delta\)-divisions. Here is an elementary one:
\begin{lemma}\label{ellipsoiddivision}
  Let \(Y\) be a subset of a bounded metric space \(X\) and suppose that \(y_1\)
  and \(y_2\) are points in \(Y\) of maximal distance. Given \(\delta\)
  with \(0<\delta < 1/2\) let \(f = (1-2\delta)/(1+2\delta)\) and let
  \(Y_1\) consist of the points \(y \in Y\) satisfying \(f d(y, y_1)
  \le d(y, y_2)\). Similarly, let \(Y_2\) consist of the points \(y
  \in Y\)  satisfying \(f d(y, y_2) \le d(y,y_1)\). Then \(\{Y_1,Y_2\}\)
  is a \(\delta\)-division of \(Y\).   
\end{lemma}
\begin{proof}
  Let \(x \in X\) and let \(r = r(Y)\) be the relative radius of \(Y\). By
  symmetry we may without loss of generality assume that \(d(x,y_1)
  \le d(x,y_2)\).  
  We will show that if \(z \in B(x,\delta r) \cap Y\), then \(z \in
  Y_1\), that is, that \(fd(z,y_1) \le d(z,y_2)\). Since the radius of
  \(Y\) is smaller than or equal to the diameter \(d(y_1,y_2)\) of
  \(Y\) it suffices to show that \(d(x,z) \le \delta d(y_1,y_2)\)
  implies that \(fd(z,y_1) \le d(z,y_2)\). However, since 
  \[d(y_1,y_2) \le d(y_1,x) + d(x,y_2) \le 2d(x,y_2),\] 
  we have 
  \begin{align*}
    d(z,y_1) \le d(z,x) + d(x,y_1) 
    \le \delta d(y_1,y_2) + d(x,y_2) 
    \le (2\delta + 1) d(x,y_2)
  \end{align*}
  and
  \begin{align*}
    d(x,y_2) \le d(x,z) + d(z,y_2) 
    \le \delta d(y_1,y_2) + d(z,y_2) 
    \le 2\delta d(x,y_2) + d(z,y_2) 
  \end{align*}
  Since \(f = (1-2\delta)/(1+2\delta)\) this gives
  \begin{displaymath}
    fd(z,y_1) \le (1-2\delta) d(x,y_2) \le d(z,y_2).
  \end{displaymath}
\end{proof}

Given a bounded metric space \(X\), a method for \(\delta\)-division
and \(r \ge 0\), we construct in Algorithm~\ref{divisive_algorithm} a
\(\delta\)-divisive cover \(\UU^r\) of \(X\) of resolution \(r\). 
Thus the persistent homology of \((\UU^r)_{s \ge r}\) is \(\delta\)-interleaved with the persistent homology of \((\CC)_{s \ge r}\) .

\begin{algorithm} \label{divisive_algorithm}
    \SetKwInOut{Input}{Input}
    \SetKwInOut{Output}{Output}
    
	\Input{A bounded metric space \(X\), a method for \(\delta\)-division and \(r \ge 0\)}
	\Output{A \(\delta\)-divisive cover \(\UU^r\) of $X$}
	\caption{Divisive cover algorithm}
	$X_0 = X$ \\
	Create list L = $\{0\}$ \\ 
 	$i = 0$ \\
	\While {There exists a $j \in L$ such that \(r(X_j) > r\) }
		{      
		$k = \argmax_{j \in L} \{\text{diameter of \(X_j\)}\}$ \\
		Construct a \(\delta\)-division \((X_{i+1},X_{i+2})\) of \(X_k\)\\

		remove $k$ from $L$ and add \(i+1\) and \(i+2\) to \(L\) \\
				
		$i = i + 2$ 
		}
         \(\UU^r = \{ X_0, X_1, \dots , X_{i}\}\)
\end{algorithm}



\section{Complexity of the Divisive Cover Algorithm}\label{sec:complexity}

For the study of complexity of Algorithm~\ref{divisive_algorithm} we
will restrict attention to the situation where \(X\) is a finite
subset of \(\R^d\) with the
\(L^{\infty}\)-metric \(d_{\infty}\). For \(1 \le i \le d\), we let
\(\pr_i \colon \R^d \to \R\) be the coordinate projection taking
\((v_1,\dots,v_d) \in \R^d\) to \(v_i\).
\begin{lemma}[Decision division]\label{decisiondivision}
  Let \(X\) be a finite subset of \(\R^d\) equipped with the
  \(L^{\infty}\)-metric \(d_{\infty}\) and let
  \(x_1\) and \(x_2\) be points in \(X\) of maximal distance. Choose a
  coordinate projection \(\pr_i\) so that \(d_{\infty}(x_1,x_2) =
  |\pr_i(x_1 - x_2)|\). Given
  \(\delta\) with \(0<\delta < 1\) let \(X_1\) consist of the points
  \(x \in X\) satisfying \(|\pr_i(x_1-x)| \le
  \frac{1+\delta}{2}|\pr_i(x_1-x_2)|\).
  Similarly, let \(X_2\) consist of the points
  in \(x \in X\) satisfying \(|\pr_i(x_2-x)| \le
  \frac{1+\delta}{2}|\pr_i(x_1-x_2)|\). Then
  \((X_1,X_2)\) is a \(\delta\)-division of \(X\).   
\end{lemma}
\begin{proof}
  Let \(p \in X\) and let \(r\) be the relative radius of
  \(X\). Note that \(d(x_1,x_2) = 2r\) in the situation of the
  asserted statement. We 
  have to show that the intersection of \(X\) 
  with the ball centered in \(p\) of radius \(\delta r\) is
  contained in one of \(X_1\) and \(X_2\). 
  Let us for convenience write \(y_1 = \pr_i(x_1)\) and \(y_2 =
  \pr_i(x_2)\) and let us assume that \(y_1 < y_2\).
  It suffices by
  construction to show that the interval
  \([\pr_i(p)-r,\pr_i(p)+r]\) is contained in one of the intervals
  \([y_1, y_1 + (1+\delta)(y_2-y_1)/2]\) and  
  \([y_2 -(1+\delta)(y_2-y_1)/2, y_2]\). This follows from the fact
  that the intersection \([y_2- (1+\delta)(y_2-y_1)/2,
  y_1+(1+\delta)(y_2-y_1)/2]\) of these intervals has length
  \(\delta(y_2-y_1) = 2 r \delta\). 
\end{proof}
\begin{theorem}\label{complexitythm}
  Let \(X\) be a finite subset of \(\R^d\) equipped with the
  \(L^{\infty}\)-metric \(d_{\infty}\) and let \(t > 0\). If \(X\) has
  cardinality \(n\), then the cover
  \(\mathcal V\) of \(X\) obtained from
  Algorithm~\ref{divisive_algorithm} is constructed in
  \(O({2^{kd}}dn)\) time, where \(k = \lceil 
  \log_{\frac{1+\delta}{2}} (t/r) \rceil\). The size of the cover
  \(\mathcal V\) is at most \(2^{kd}\). The nerve of \(\mathcal V\)
  can be constructed in \(O(2^{2^{kd}}dn)\) time.
\end{theorem}
Note that for fixed \(d\) and \(\delta\), the term \(2^{kd}\) is
polynomial in the ratio \(r/t\) between the radius \(r\) of \(X\) and
the threshold radius \(t\).

Let \(\mathcal V\) be as in Theorem~\ref{complexitythm}.
Given \(s \ge t\) we write \({\mathcal V}_s\) for the
cover of \(X\) given by members of \(\mathcal V\) of radius less than
\(s\). By construction,  for \(s \ge t\), the inclusion of \({\mathcal V}_s\) in 
\(\UU_s\) is cofinal. Thus 
by Lemma \ref{cofinality}, for filtration values
greater than $t$, the persistent homology of the cover \(\mathcal V\)
coincides with the persistent homology of \(\UU\). 

\begin{proof}[Proof of Theorem~\ref{complexitythm}]
Note that in
the \(L^{\infty}\)-metric, the radius of a subset of \(\R^d\) is given
by the maximum of the radii of its coordinate projections to
\(\R\). 
A \(\delta\)-decision division (\ref{decisiondivision}) reduces the
radius of this coordinate projection 
by the factor \(\frac{1+\delta}{2}\). Thus the radius of any \(d\)-fold  \(\delta\)-divided part of
\(X\) is at most \(r \left(\frac{1+\delta}{2}\right)\), where \(r\) is
the radius of \(X\).
If we let \(k = \lceil \log_{\frac{1+\delta}{2}} (t/r)
\rceil\), then the radius of any \({kd}\)-fold \(\delta\)-divided part of \(X\)
is at most \(r(\frac{1+\delta}{2})^k \le t\).
Since 
each \(\delta\)-decision division consists of two parts, we conclude that
\(\mathcal V\) can be produced 
by making at most
\(2^{kd}\) \(\delta\)-decision divisions. 
Since we work in the \(L^{\infty}\) metric, extremal points can be
found by computing min- and max-values for
the coordinate projections of points in \(X\).
Similarly \(\delta\)-decision division can be made by computing min- and max-values for
the coordinate projections of points in \(X\). Each of these steps require \(O(nd)\)
time, so the cover is of size at most \(2^{kd}\) and it can be
constructed in \(O(2^{kd}nd)\) time.

Finally, the nerve of the cover \(\mathcal V\) 
is constructed by calculating intersections
of members of \(\mathcal V\). Calculating the intersection of \(i \le d\)
subsets of \(X\) can be done by, for each element \(x\) of 
\(X\), deciding if \(x\) is a member of the intersection. The
complexity of this is \(O(ni)\). Since
the cardinality of \(\mathcal V\) is at most \(2^{kd}\), independently of
\(n\), the time of calculating the nerve is \(O(2^{2^{kd}}n)\).
\end{proof}

We shall use the following result to show that a \(\delta\)-decision
division of \(X \subseteq \R^d\) gives a \(d^{-1/p} \delta\)-divisive
cover of \(X\) in the \(L^{p}\)-metric.
This stems from the fact that all \(L^{p}\)-metrics are equivalent.  

\begin{proposition}\label{squeeze}
Let \(d_1\) and \(d_2\) be metrics on \(X\) and let
\(\alpha\) and \(\beta\) be positive numbers such that for all \(x, y
\in X\) the 
inequality 
\begin{equation*}
\alpha d_1(x, y) \le d_2(x, y) \le \beta d_1(x, y)
\end{equation*}
holds. 
Then every \(\delta\)-filtered cover of \((X, d_1)\) is a 
\(\delta \alpha/\beta\)-filtered cover of \((X, d_2)\).
\end{proposition}
\begin{proof}
We emphasize the metrics \(d_1\) and \(d_2\) in the notation by
writing \(\UU^{d_1}_t\) and \(\UU^{d_2}_t\) for the covers of \(X\) consisting of members
of \(\UU\) 
contained in a closed ball of radius \(t\) in \((X, d_1)\) and \((X, d_2)\) respectively. 

By assumption, there are inclusions of balls
\begin{equation*}
B_{d_1}\left(x, t/\beta \right) \subseteq 
B_{d_2}\left(x, t \right) \subseteq
B_{d_1}\left(x, t/\alpha \right), 
\end{equation*}
so
\begin{equation*}
\UU^{d_1}_{t/\beta} \subseteq 
\UU^{d_2}_{t} \subseteq
\UU^{d_1}_{t/\alpha}.
\end{equation*}

Given a point \(x \in X\) and a radius \(t > 0\), we can find a set
\(A \in \UU^{d_1}_{t/\beta}\) such that \(B_{d_1}(x,
t \delta /\beta) \subseteq A\) since  \(\UU\) is
\(\delta\)-filtered in \((X, d_1)\). Due to the above inclusions,
\(A\) is also in \(\UU^{d_2}_{t}\) and \(B_{d_2}(x, t \delta
\alpha/\beta) \subseteq B_{d_1}(x, t \delta /\beta)
\subseteq A\). Thus \(\UU\) is an \(\delta \alpha/\beta\)-filtered
cover of \((X,d_2)\).
\end{proof}

In the case where \(d_1\) is the \(L_{\infty}\)-metric and \(d_2\) is the
\(L_p\)-metric on \(\R^d\) the inequalities in
Proposition~\ref{squeeze} hold for \(\alpha = 1\) and
\(\beta = d^{1/p}\). Thus, if \(\UU\) is a \(\delta\)-filtered cover of
\(X\) with respect to the \(L_{\infty}\)-metric, then it is a
\(d^{-1/p}\delta\)-filtered cover of \(X\) with respect to the
\(L_p\)-metric. In particular it is \(\delta/\sqrt{d}\)-filtered with
respect to the Euclidean metric.


\section{Examples} \label{sec:examples}

\subsection{Generated data}

\subsubsection{Sphere}
We used divisive cover with the \(\delta\)-division of Lemma \ref{ellipsoiddivision}
to calculate the persistent homology of a generated sphere. We
generated 1000 data points with a radius normally distributed with a
mean of 1 and a standard deviation of 0.1 and uniform angle. The top
panel of Figure~\ref{figure} shows the resulting persistence barcodes.  

\subsubsection{Torus}
We calculated the persistent homology of a generated torus using
divisive cover with the \(\delta\)-division of Lemma \ref{ellipsoiddivision}. We generated 400 data points on a torus. The torus was
generated as the product space of 20 points each on two circles of
radius 1 with uniformly distributed angles. The second panel of
Figure~\ref{figure} shows the  persistence barcodes of the generated
torus. 

\begin{figure} \label{figure}
\includegraphics[width = \textwidth]{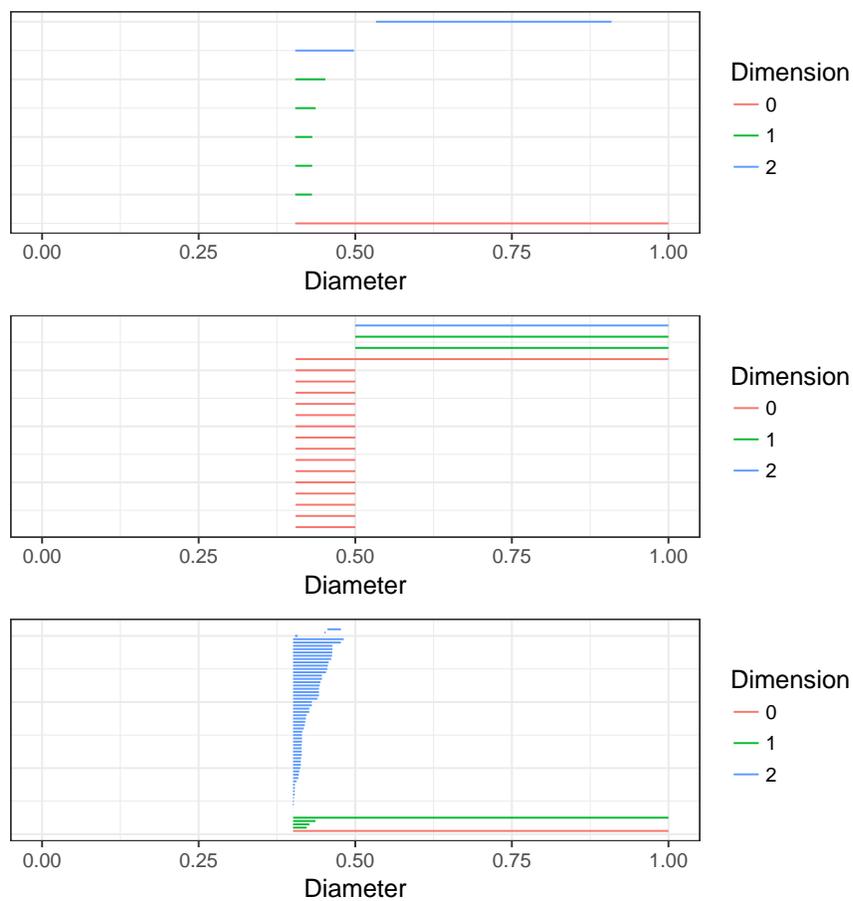}
\caption{Persistence barcodes using divisive cover. All barcodes are
  shown for relative diameter between \(0.4\) and \(1\). The first
  panel shows the persistence barcodes of a sphere using divisive
  cover with $\delta = 0.05$ and the second panel shows persistence
  barcodes of a torus with $\delta = 0.06$. The third panel shows
  persistence barcodes of $X(300, 30)$, with $\delta = 0.025$} 
\end{figure}

\subsection{Natural images}
The space of 3 by 3 high-contrast patches of natural images has been
analysed using witness complexes before \citep{Carlsson2008}. The
authors analysed high-density subsets of 50,000 random 3 by 3 patches
from a collection of $4 \times 10^6$ patches presented in \cite{Hateren1998}. They
denote the space $X(k, p)$ of $p$ percent highest density patches
using the $k$-nearest neighbours to estimate density and find that
$X(300,30)$ has the topology of a circle. We repeat this analysis
using divisive cover with the \(\delta\)-division of Lemma
\ref{ellipsoiddivision} and show that calculating persistent homology
without landmarks is possible for real world data sets. The bottom
panel of Figure~\ref{figure} show the persistence barcodes of $X(300,
30)$.  


\section{Conclusion} \label{sec:conclusion}

Filtered covers as the underlying structure for filtered complexes
provides new insights into topological data analysis. It can be used
as a basis for new constructions of simplicial complexes that are
interleaved with the {\v{C}}ech nerve. We are not aware of any
previous literature that made use of covers in such a way. Divisive
covers are just one possible way to create \(\delta\)-filtered
covers. Many other constructions are available, for example optimized
versions of the \(\delta\)-filtered {\v{C}}ech cover we have
presented.

The idea of a divisive cover is conceptually simple and easy to
implement. Compared to the {\v{C}}ech nerve, the nerve of a divisive
cover can be substantially smaller. On the other hand, the witness
complex is often considerably smaller than the divisive cover complex.
Although we give theoretical guarantees that are linear in \(n\), in
practice persistent homology calculations using the divisive cover algorithm
proposed here are not competitive with state of the art
approximations to the Vietoris-Rips complex \cite{MR3394707,simba}.
We see divisive covers as a new class of simplicial complexes that can
be studied in a fashion similar to Vietoris-Rips filtrations. It is
possible to reduce the size of the divisive cover complex, for example
by using landmarks.  We did not address such improvements in the
present paper. It might also be possible to combine a version of the
Vietoris-Rips complex for low filtration values and a version of the
divisive cover complex for high filtration values. The version of
divisive cover we have presented is easy to implement and performs
well at large filtration values.


\section*{Acknowledgements}
This research was supported by the Research Council of Norway through grant 248840. 


\bibliographystyle{plainnat}

\end{document}